\documentclass[a4paper]{amsart}
\pdfoutput=1
\usepackage[T1]{fontenc}

\usepackage{amsmath, amssymb, amsthm}
\usepackage{booktabs}
\usepackage{tikz}
\usepackage{wrapfig}

\usepackage{xspace}
\newcommand{\mptopcom}{\texttt{mptopcom}\xspace}
\newcommand{\TOPCOM}{\texttt{TOPCOM}\xspace}
\newcommand{\polymake}{\texttt{polymake}\xspace}

\newcommand{\RR}{\mathbb{R}}

\DeclareMathOperator{\gkz}{gkz}
\DeclareMathOperator{\vol}{vol}
\DeclareMathOperator{\conv}{conv}

\usepackage[kerning,spacing,tracking]{microtype}

\theoremstyle{theorem}
\newtheorem{theorem}{Theorem}
\theoremstyle{plain}
\newtheorem{proposition}[theorem]{Proposition}
\newtheorem{lemma}[theorem]{Lemma}
\theoremstyle{definition}
\newtheorem{example}[theorem]{Example}
\theoremstyle{remark}
\newtheorem{remark}[theorem]{Remark}

\usepackage{color}
\usepackage[naturalnames=false]{hyperref}
\hypersetup{
  colorlinks = true,          
  urlcolor = blue, linkcolor = blue, citecolor = blue
}

\usepackage[backend=biber]{biblatex}
\title{Regular Flips in \mptopcom}
\author{Lars Kastner}
\address{
Technische Universit\"at Berlin\\
Chair of Discrete Mathematics/Geometry\\
Stra\ss e des 17.~Juni 136\\
10623 Berlin\\
Germany
}
\email{kastner@math.tu-berlin.de}

\begin{document}

\begin{abstract}
A triangulation of a point configuration is regular if it can be given by a
height function, that is every point gets lifted to a certain height and
projecting the lower convex hull gives the triangulation. Checking regularity of
a triangulation usually is done by solving a linear program. However when
checking many flip-connected triangulations for regularity, one can instead ask
which flips preserve regularity. When traversing the flip graph for enumerating
all regular triangulations, this allows for vast reduction of the linear
programs needing to be solved. At the same time the remaining linear programs
will be much smaller.
\end{abstract}
\maketitle

\section{Introduction}

Enumerating regular triangulations is an important problem in many areas of
mathematics, such as tropical geometry \cite{schlaefli,cubic_surfaces}, optimization
\cite[Section~1.2]{Triangulations}, and mathematical physics \cite{ftheory}.
There are two main software systems for enumerating triangulations, namely
\TOPCOM \cite{TOPCOM-paper} and \mptopcom \cite{mptopcom-paper} which is a
fusion of the former and \polymake \cite{DMV:polymake}. This paper will provide
details on recent developments of \mptopcom.

The algorithm of \mptopcom is reverse search \cite{reverse-search}. We will
leave out most details, but essentially reverse search on the flip graph
associates to every triangulation a unique predecessor. In our case this is the
flip neighbor with lexicographically largest GKZ-vector. Taking only the edges
from such a predecessor relation in the flip graph results in the so-called
\emph{reverse search tree}.

The two main advantages of reverse search are that reverse search is
output-sensitive, meaning it only consumes a constant amount of memory and that
it can be parallelized, \mptopcom uses the \emph{budgeted reverse search}
\cite{mplrs}. Further explanation of \mptopcom's internal workings can
be found in \cite{mptopcom-paper} and on the \mptopcom
webpage\footnote{\url{https://polymake.org/mptopcom}}.
Note that parallelization has also been adopted in \TOPCOM.

To enumerate regular triangulations exclusively, one can restrict to regular
triangulations, but still allowing all flips in between. Second, we could
restrict to regular flips only.
The second approach results in a different reverse search tree, as it considers
the subgraph of regular flips of the flip graph.
Thus the goal of this paper is to mitigate the main bottleneck of solving
regularity linear programs (LP):
First, we want to reduce the size of the LPs appearing.
Second, we want to reduce the number of LPs that need to be solved
altogether.

\section{Regular Flips}

In this section we will carve out a description of regular flips between
regular triangulations in terms of GKZ-vectors. We follow the notation and
conventions of the triangulation bible \cite{Triangulations}.

Start with a point configuration $A = \{a_1, \ldots, a_n\}$
and let $T$ be a triangulation of $A$. Recall that the \emph{GKZ-vector}
$\gkz(T) \in \RR^n$ of $T$ is defined as
\[
   \gkz(T)_i = \sum_{s \ni a_i, s\in T} \vol(s).
\]
Now for a flip $f:T\leadsto T'$ we define $\gkz(f):=\gkz(T')-\gkz(T)$.

\begin{example}
   Consider the two triangulations of the standard square $[0,1]^2$ and the
   flip between them.
   \[
      \begin{array}{cccc}
         f: & \quad
            \begin{tikzpicture}[baseline=.4em,scale=.5]
               \node[anchor=east] at (0,0) {$0$};
               \node[anchor=east] at (0,1) {$3$};
               \node[anchor=west] at (1,0) {$1$};
               \node[anchor=west] at (1,1) {$2$};
               \draw (0,0) -- (1,0) -- (1,1) -- (0,1) -- (0,0) -- (1,1);
            \end{tikzpicture}\quad
            & \leadsto & \quad
            \begin{tikzpicture}[baseline=.4em,scale=.5]
               \node[anchor=east] at (0,0) {$0$};
               \node[anchor=east] at (0,1) {$3$};
               \node[anchor=west] at (1,0) {$1$};
               \node[anchor=west] at (1,1) {$2$};
               \draw (0,1) -- (0,0) -- (1,0) -- (1,1) -- (0,1) -- (1,0) ;
            \end{tikzpicture}\\
            & (2,1,2,1) & (-1,1,-1,1) & (1,2,1,2)
      \end{array}
   \]
   The second row contains the GKZ-vectors of the triangulations and the flip
   GKZ-vector between them.
\end{example}

To decide whether $f$ is regular from $\gkz(f)$ it is important
that two flips $f_0:T\leadsto T'$ and $f_1:T\leadsto T''$ have different
GKZ-vectors. For an idea why this could be true, take the following
proposition.

\begin{proposition}\label{prop:flip_gkz}
   Let $f:T\leadsto T'$ be a flip. Then $\gkz(f)\not=0$.
\end{proposition}
\begin{proof}
   It is already clear that $\gkz(f)$ will have zero entries in all
   coordinates, but those that belong to the corresponding circuit.
   Next Lemma~2.4.2 of \cite{Triangulations} states that the two different
   triangulations of a circuit are regular. But two distinct regular
   triangulations cannot have the same GKZ-vector.
\end{proof}

Every flip arises from a \emph{corank-one configuration} \cite[2.4.1]{Triangulations},
that is a subset $J\subseteq \{1,\ldots,n\}$ such that $\conv(\{a_j|j\in J\})$ is
full-dimensional and there exists a unique, non-trivial affine dependence
relation $\sum_{j\in J}\lambda_j a_j=0$ with $\sum_{j\in J}\lambda_j=0$, of
course up to scalar multiple. The set $J$ can now be divided into three parts:
\[
   J_+\:=\ \{j\in J\ \vert\ \lambda_j>0\},\
   J_0\:=\ \{j\in J\ \vert\ \lambda_j=0\},\
   J_-\:=\ \{j\in J\ \vert\ \lambda_j<0\}.
\]
A corank-one configuration has exactly two triangulations
\cite[Lem.~2.4.2]{Triangulations}, namely
\[
   T_+\ =\ \{J\setminus \{j\}\ \vert\ j\in J_+\},\mbox{ and }
   T_-\ =\ \{J\setminus \{j\}\ \vert\ j\in J_-\}.
\]
We say that a triangulation $T$ of $A$ is \emph{compatible} with a corank-one
configuration $J$ if it contains the simplices of the triangulation $T_+$.
Replacing these by the simplices of $T_-$ is exactly what it means to apply a
flip. The pair $(J_+, J_-)$ is (the \emph{Radon partition} of) a circuit. Note
that a circuit can give rise to at most one flip:
\begin{lemma}\label{lemma:circuit_lift}
   Assume that $J$ is a corank-one configuration with circuit $(J_+,J_-)$
   compatible with the triangulation $T$. Then there cannot be a second
   corank-one configuration compatible with $T$ and the same circuit
   $(J_+,J_-)$.
\end{lemma}
\begin{proof}
   Write $J=J_+\cup J_0\cup J_-$ and assume that there is $J'=J_+\cup J_0'\cup
   J_-$ with $J_0\not= J_0'$.
   Denote the (sub-)triangulations associated to $J$ by $T_+$ and $T_-$, and
   those of $J'$ by $T_+'$ and $T_-'$.
   Each of these corank-one configurations give rise to a flip we can apply to
   $T$.
   By assumption, no maximal simplex of $T_+$ is part of $T_+'$ and vice versa,
   since every maximal simplex of $T_+$ contains $J_0$ and of $T_+'$ contains
   $J_0'$.
   Now if we apply the flip associated to $J$, the maximal simplices of $T_+'$
   remain untouched. In particular, every simplex of $T_+'$ will still have a
   simplex corresponding to $J_-$ as a face. On the other hand, $T_-$ got
   inserted and here every maximal simplex has a face corresponding to $J_+$.
   However
   \[
      \conv(\{a_j\ \vert\ j\in J_+\})\cap\conv(\{a_j\ \vert\ j\in J_-\})\ \not=\ \emptyset,
   \]
   stemming from the affine dependence equation, while $J_+\cap J_-=\emptyset$.
   That means that this intersection of faces does not form a face of both and
   hence is invalid. Thus such $J$ and $J'$ cannot simultaneously exist.
\end{proof}

Let us investigate the GKZ-vector of a flip. We already deduced that it
can only have non-zero entries in the indices corresponding to the corank-one
configuration $J$. Thus we can write the subvector of $\gkz(f)$ in the indices
of $J$ as
\[
   \gkz(f)_J\ =\ \gkz(T_-)-\gkz(T_+).
\]
Next, one deduces that the points $a_j$ for $j\in J_0$ are part of every simplex
for both $T_+$ and $T_-$, hence for these points we have $\gkz_j(f)=0$.
Now take $j\in J_+$. Then we have
\[
   \gkz_j(f)\ =\ \gkz_j(T_-)-\gkz(T_+)\ =\ \sum_{a_j\in S\in T_-}\vol(S)-\sum_{a_j\in S\in T_+}\vol(S).
\]
By construction though, $a_j$ is part of every simplex of $T_-$, so the first
part is just the volume of the convex hull of the corank-one configuration,
write $VJ=\vol(\conv\{a_i\ \vert\ i\in J\})$.
There is exactly one maximal simplex in $T_+$ that $a_j$ is not part of, namely
stemming from $J\setminus \{j\}$. Thus we get
\[
   \begin{array}{rcl}
      \gkz_j(f) & = & VJ - (VJ-\vol(\conv\{a_i\ \vert\ i\in J\setminus\{j\}\}))\\
      & = & \vol(\conv\{a_i\ \vert\ i\in J\setminus\{j\}\})\ >\ 0.
   \end{array}
\]
Similarly we get $\gkz_j(f)<0$ for $j\in J_-$.
Thus $\gkz(f)$ uniquely determines the corresponding circuit and
using Lemma~\ref{lemma:circuit_lift} a unique corank-one
configuration. Hence given $\gkz(f)$ we can uniquely identify the corresponding
flip $f$. In particular, for two flips $f$ and $f'$ of $T$, $\gkz(f)$ cannot be
a scalar multiple of $\gkz(f')$.

Note that for a fixed regular triangulation $T$, all the vectors $\gkz(f)$ for $f:T\leadsto
T'$ form the edge cone of the secondary polytope at the vertex $\gkz(T)$. The
dual of this cone is $\sec(T)$. With this in mind we can state the
following theorem.

\begin{theorem}\label{thm:regular_flip}
   Let $f:T\leadsto T'$ be a flip between two triangulations, with $T$ being
   regular. If $\gkz(f)$ forms a facet of $\sec(T)$, then $T'$ is regular.
   Equivalently, if $\gkz(f)$ generates an extremal ray of the edge cone of the
   secondary polytope at $T$, then $T'$ is regular.
\end{theorem}
\begin{proof}
   Assume $\gkz(f)$ forms a facet, then dually it corresponds to an edge of the
   secondary polytope. By the discussions above it is the unique flip with that
   edge direction. Hence it is the unique flip to the vertex neighboring
   $\gkz(T)$ in direction $\gkz(f)$. But vertices of the secondary polytope
   correspond to regular triangulations. Hence $T'$ must be regular.
\end{proof}

Usually the linear program we need to solve to determine whether $\gkz(f)$
forms a facet of $\sec(T)$ is much smaller than the LP for determining whether
$T$ is regular, finishing our goals of reducing the size
of the LPs we need to solve.

\section{Caching}
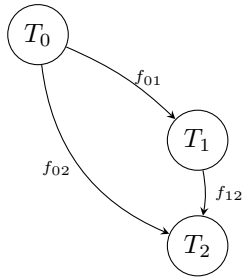
\begin{figure}
   \begin{center}
      \begin{tikzpicture}[scale=.7]
         \node[draw, circle] (t0) at (0,0) {$T_0$};
         \node[draw, circle] (t1) at (3,-2) {$T_1$};
         \node[draw, circle] (t2) at (3,-4) {$T_2$};
         \path[-stealth] (t0) edge [bend right = 30] node[anchor=east, font=\tiny] {$f_{02}$} (t2);
         \path[-stealth] (t0) edge [bend right = -10] node[anchor=west, font=\tiny] {$f_{01}$} (t1);
         \path[-stealth] (t1) edge [bend right = -10] node[anchor=west, font=\tiny] {$f_{12}$} (t2);
      \end{tikzpicture}
   \end{center}
   \caption{Flips between three triangulations}\label{fig:broken-cache}
\end{figure}
At this point we want to illustrate how the new approach interferes with our
previous caching approach \cite[Sec.~7.4]{mptopcom-paper}, where we would cache
whether a triangulation was regular, and consider a flip valid whenever the
target triangulation is regular. Using this caching approach can result in
mixing in non-regular flips with the regular flips that we want to restrict to.
For a flip $f_{ij}:T_i\leadsto T_j$
denote the reverse flip as $f_{ji}:T_j\leadsto T_i$.

To further explain this example, we need the notion of up- and downflips, which
for the sake of brevity will just be the vertical direction of the flip in
Figure~\ref{fig:broken-cache}. A flip is part of the reverse search
tree if its reverse is the largest regular upflip from the target triangulation, see
\cite{mptopcom-paper} for further details.

In Figure~\ref{fig:broken-cache}, assume that the $T_i$ are regular
triangulations, and that $f_{01}$ and $f_{12}$ are regular flips, while
$f_{02}$ is not. Furthermore assume that regularity of $T_2$ was not cached
when we were in $T_0$. Then $f_{02}$ was correctly identified as non-regular
and thus not used. But when we arrive in $T_2$ via $f_{12}$, $f_{20}$ is
considered valid, since caching determines that $T_0$ is regular and actual
regularity of $f_{20}$ is never checked. Since $f_{02}$ is the largest upflip
from $T_2$, the algorithm incorrectly determines that $f_{12}$ is not part of
the reverse search tree and $T_2$ is never visited.

Note that this is a very subtle phenomenon, as it can only occur for point
configurations that are large enough to have non-regular flips. Furthermore
these non-regular flips need to be the largest upflips to actually cause
trouble. The smallest point configuration exhibiting these properties was
$\Delta_2\times\Delta_5$.

\section{Avoiding Linear Programs}

The edge cone of the secondary polytope at $\gkz(T)$ is spanned by the
$v^1,\ldots, v^n$, with $v^i=\gkz(f_i)$, where $f_i$ is the $i$-th flip of the
triangulation $T$. We assume that $T$ is regular, hence this cone is pointed. None of the $v^i$ can be scalar multiples of each other, due to
Lemma~\ref{lemma:circuit_lift}.
Both conditions are necessary for the following tricks to work.
According to Thm.~\ref{thm:regular_flip}, we need to find the extremal rays among the
$v^i$, translating to solving the following linear program for every
$i=1,\ldots,n$:
\[
   v^i\ =\ \sum_{j\not=i} x_j\cdot v^j,\mbox{ with } x_j\ge 0.
\]
Thus, for every regular triangulation $T$ we have to solve $n$
LPs, where $n$ is the number of flips from $T$.

One could try to compute the rays of this cone directly, but
due the high dimension this has its own set of challenges \cite{polymake:2017}.
Nevertheless, considering all flips at the same time allows us to preprocess this system.
As remarked in the proof of Prop.~\ref{prop:flip_gkz}, the only non-zero
entries of a flip GKZ-vector corresponds to the circuit giving rise to
the flip.
Thus most entries of $v^i$ will be zero, unless in boundary case point
configurations. Hence the system we are considering is quite sparse and there
are a few simple tricks we can apply to spot some rays immediately.

\begin{figure}
   \begin{center}
      \tiny
   \begin{tabular}{r|rrrrrrrrrrrrrrrrrrr}
      \toprule
      & $0$ & $1$ & $2$ & $3$ & $4$ & $5$ & $6$ & $7$ & $8$ & $9$ & $10$ & $11$ & $12$ & $13$ & $14$ & $15$ & $16$ & $17$\\
      \midrule
      $v^{1}$ & $0$ & $0$ & $0$ & $0$ & $0$ & $0$ & $0$ & $0$ & $-3$ & $3$ & $0$ & $0$ & $0$ & $0$ & $3$ & $-3$ & $0$ & $0$\\
      $v^{2}$ & $0$ & $0$ & $0$ & $0$ & $0$ & $0$ & $-1$ & $0$ & $0$ & $0$ & $0$ & $1$ & $1$ & $0$ & $0$ & $0$ & $0$ & $-1$\\
      $v^{3}$ & $-1$ & $0$ & $0$ & $0$ & $1$ & $0$ & $1$ & $0$ & $0$ & $0$ & $-1$ & $0$ & $0$ & $0$ & $0$ & $0$ & $0$ & $0$\\
      $v^{4}$ & $0$ & $0$ & $1$ & $0$ & $-1$ & $0$ & $0$ & $0$ & $-1$ & $0$ & $1$ & $0$ & $0$ & $0$ & $0$ & $0$ & $0$ & $0$\\
      $v^{5}$ & $-1$ & $1$ & $0$ & $0$ & $0$ & $0$ & $1$ & $0$ & $0$ & $0$ & $0$ & $-1$ & $0$ & $-1$ & $0$ & $0$ & $0$ & $1$\\
      $v^{6}$ & $0$ & $0$ & $0$ & $0$ & $-1$ & $1$ & $0$ & $0$ & $0$ & $0$ & $0$ & $0$ & $0$ & $0$ & $0$ & $0$ & $1$ & $-1$\\
      $v^{7}$ & $0$ & $0$ & $-3$ & $3$ & $0$ & $0$ & $0$ & $0$ & $3$ & $-3$ & $0$ & $0$ & $0$ & $0$ & $0$ & $0$ & $0$ & $0$\\
      $v^{8}$ & $0$ & $0$ & $0$ & $0$ & $0$ & $0$ & $0$ & $1$ & $0$ & $0$ & $-1$ & $0$ & $0$ & $-1$ & $0$ & $0$ & $1$ & $0$\\
      $v^{9}$ & $0$ & $0$ & $0$ & $-1$ & $1$ & $0$ & $0$ & $0$ & $0$ & $1$ & $0$ & $-1$ & $0$ & $0$ & $0$ & $0$ & $-1$ & $1$\\
      $v^{10}$ & $0$ & $0$ & $0$ & $0$ & $0$ & $0$ & $1$ & $0$ & $0$ & $-1$ & $0$ & $0$ & $-1$ & $0$ & $0$ & $1$ & $0$ & $0$\\
      $v^{11}$ & $0$ & $-1$ & $1$ & $0$ & $0$ & $0$ & $0$ & $0$ & $0$ & $0$ & $0$ & $0$ & $0$ & $1$ & $-1$ & $0$ & $0$ & $0$\\
      $v^{12}$ & $1$ & $-1$ & $0$ & $0$ & $0$ & $0$ & $-1$ & $0$ & $0$ & $0$ & $1$ & $0$ & $0$ & $1$ & $0$ & $0$ & $-1$ & $0$\\
\bottomrule
   \end{tabular}
      \normalsize
   \end{center}
   \caption{An example set of flip GKZ-vectors for a regular triangulations of $\Delta_2\times\Delta_5$. The corresponding triangulation can be found at our GitHub repository.}
   \label{fig:flip_matrix}
\end{figure}

We present these in a simplified form, where $v^1$ is the pivot. Note that by a
convex combination inside of a cone we mean a positive linear combination of
the generators. We will use the term ray to mean both ray and ray generator.

\begin{proposition}\label{prop:one_pos}
   Assume that $v^1[0]>0$ and $v^i[0]=0$ for all other $i\not=1$. Then $v^1$
   is a ray and the subsystem $v^2,\ldots,v^n$ can be solved
   without including $v^1$.
\end{proposition}
\begin{proof}
   Every convex combination using $v^1$ will have first coordinate $>0$, so it
   cannot give $v^2,\ldots,v^n$. Similarly convex combinations of
   $v^2,\ldots,v^n$ result in the first entry being zero. Thus all
   vectors $v^i$ with $v^i[0]=0$ form a face of the cone.
\end{proof}
\begin{example}
   One example is $v^{6}$ in Figure~\ref{fig:flip_matrix}, it is the
   only vector with $v^{6}[5]>0$.
\end{example}

\begin{remark}
   The above proposition can be generalized such that we have
   \[
      v^1[0],\ldots,v^m[0]>0 \mbox{ and } v^{m+1}[0],\ldots,v^n[0]=0.
   \]
   However then we cannot deduce that $v^1,\ldots,v^m$ are rays, but the linear
   programs for $v^{m+1},\ldots,v^n$ become much smaller.
\end{remark}

Next we will reduce to smaller systems by including some modified vectors that
are needed to solve the smaller system correctly, but that are already known
not to be rays, so we do not need to solve a LP for these.

\begin{proposition}
   Assume that $v^1[0]>0$ and that $v^2[0]<0$, while $v^i[0]=0$ for all other
   $i>2$. Then $v^1$ and $v^2$ are rays and it suffices to solve the smaller
   system $v^3,\ldots,v^n, v^1[0]\cdot v^2- v^2[0]\cdot v^1$, where the last
   vector is not a ray of the larger system.
\end{proposition}
\begin{proof}
   Considering convex combinations it is again obvious that $v^1$ and $v^2$ are
   rays. Now for a convex combination resulting in one of the vectors
   $v^3,\ldots,v^n$, we need the first entries of $v^1$ and $v^2$ to cancel,
   resulting in the last vector. However this last vector is already a convex
   combination of $v^1$ and $v^2$.
\end{proof}

\begin{remark}\label{rem:two_sided}
   Just as with Prop.~\ref{prop:one_pos}, we can generalize the previous proposition:
   Assume that $v^1[0] > 0$, $v^2[0],\ldots, v^m[0]<0$ and $v^i[0]=0$ for all
   other $i>m$. Then $v^1$ is a ray and it suffices to solve the smaller system
   \[
      v^{m+1},\ldots,v^n,\; v^1[0]\cdot v^2-v^2[0]\cdot v^1,\;\ldots,\;v^1[0]\cdot v^m-v^m[0]\cdot v^1,
   \]
   where the last $m-1$ vectors are not rays of the larger system.
\end{remark}

\begin{example}
   There are several examples for this situation in
   Figure~\ref{fig:flip_matrix}, e.g. the columns 0, 1, 2, 3, 4, 7, 8 and 9
   work for Remark~\ref{rem:two_sided}. On the other hand, column 6 does not
   work.

   Given a column index it is clear which operation to apply and which
   subsystem to consider. Take the following sequence of column indices:
   $[5,7,16,0,1,2,8,9,6]$. After applying these operations the only rows left
   to consider are $v^3, v^5, v^9, v^{12}$, all other rows have now been
   confirmed to be rays.
\end{example}

\begin{remark}\label{remark:small_lp}
   In the case that $n=2$, where we need to determine whether $v^1$ is a ray
   and we already know that $v^2$ is not a ray of the larger system, it is
   enough to determine whether $v^2$ is a scalar multiple of $v^1$.
\end{remark}

All of these tricks can be applied recursively to the smaller subsystems. They
are much cheaper than solving actual LPs.

\begin{example}
   In Table~\ref{tab:regular_flips:n_LPs} we let \mptopcom 1.4 output the number of
   LPs it actually solved for different examples.
   We chose the four dimensional cube $I^4$ since it has become the standard
   first benchmark for enumeration of triangulations, $P_5$ from
   \cite{shortest_paths} due to its relevance in optimization, and the products
   of simplices due to their connection to matroid theory \cite{multisplits}
   and tropical geometry \cite{tropical_convexity,tropical_essentials}.
   The entry $0$ for $\Delta_2\times\Delta_5$ means that all LPs were of the type
   from in Remark~\ref{remark:small_lp}.
   An interpretation of a small number of LPs is that non-regular
   triangulations have very few flips, i.e. have high flip-deficiency
   \cite[Sec.~7.2]{Triangulations}.
\end{example}

\begin{table}[th]\centering
   \caption{Number of linear programs actually solved by \mptopcom 1.4 (single threaded) for some examples. (with \texttt{--regular --orbit-cache 40000 --flip-cache 40000})}
   \label{tab:regular_flips:n_LPs}
   \begin{tabular}{lrrrr}
      \toprule
      Example & \hspace{1cm}$I^4$ & \hspace{1cm}$P_5$ & \hspace{.3cm}$\Delta_2\times\Delta_5$ & \hspace{.3cm}$\Delta_2\times\Delta_6$\\
      \midrule
      $\#$ orbits of regular triangulations
      & $235\,277$ & $27\,248$ & $13\,621$ & $531\,862$\\
      Linear programs solved & $6\,642$ & $638$ & $0$ & $261$ \\
      \bottomrule
   \end{tabular}
\end{table}

\section{Experimental results}


\begin{table}[th]\centering
   \caption{Example timings for enumerating regular triangulations}
   \label{tab:regular_flips:experimental_results}
   \centering
   \begin{tabular}{rlrrrr}
      \toprule
      & Example & \hspace{1cm}$I^4$ & \hspace{1cm}$P_5$ & \hspace{.3cm}$\Delta_2\times\Delta_5$ & \hspace{.3cm}$\Delta_2\times\Delta_6$\\
         \midrule
         1. & \mptopcom 1.1 (20 threads)      & 1159.15 & 793.61 & 67.44 & 15786.31\\
         2. & \mptopcom 1.2 (20 threads)      & 1380.39 & 601.68 & 67.61 & 15379.18\\
         3. & \mptopcom 1.3 (20 threads)      &  991.45 & 170.66 & 52.39 &  2670.98\\
         4. & \mptopcom 1.4 (20 threads)      &   82.79 &  31.22 &  9.44 &   405.05\\
         \midrule
         5. & \TOPCOM 1.1.2 (single threaded) &  709.97 &     -- & 50.17 &  5178.04 \\
         6. & \mptopcom 1.4 (single threaded)     &  958.92 & 250.55 & 30.22 &  5327.87 \\
         7. & \mptopcom 1.4 (single threaded)     &  495.98 &  83.17 & 20.20 &  3641.78 \\
         8. & \mptopcom 1.4 (single threaded)     &  339.08 &  78.78 & 20.04 &  2688.20 \\
      \bottomrule
   \end{tabular}
   \vspace{.4em}
   
   All timings are measured in seconds.
\end{table}

Table~\ref{tab:regular_flips:experimental_results} contains a few example runs
for different settings, all run on the cluster of the mathematics department
of TU Berlin.
The examples are the same as in Table~\ref{tab:regular_flips:n_LPs}.
The runs 1 and 2 are almost the same except for small outliers caused
by other jobs running in parallel on the cluster.
We ran \TOPCOM with the option \verb|--affinesymmetries|, since affine
symmetries are assumed for \mptopcom anyway.
For $P_5$ we only counted \emph{central} regular triangulations, i.e.
triangulations such that all simplices have the first point of the input as a
vertex. This option does not exist for \TOPCOM to the best of our knowledge.
The default cache settings in run 6 of \mptopcom are not enough to beat \TOPCOM, thus
for run 7 we increased the cache size of all three caches to $100\,000$ and for
run 8 to $200\,000$.

\begin{remark}
   Our GitHub repository contains some valgrind profiles
   for of \mptopcom 1.2, 1.3, and 1.4 running on $I^4$. While
   \mptopcom 1.2 and 1.3 spend 88\% of the runtime checking regularity, this
   number is only 43\% for \mptopcom 1.4. Furthermore \mptopcom 1.4 spends only
   1\% of the time solving LPs, while \mptopcom 1.3 spends 79\% and
   \mptopcom 1.2 spends 32\% of the time on LPs. The latter difference is due
   to \mptopcom 1.2 having to assemble much larger LPs for regularity of
   triangulations, while \mptopcom 1.3 only solves flip regularity LPs, but a
   much larger number of these.
\end{remark}
Supplementary material, such as cluster scripts, our input files used and the
benchmark logs, is available on GitHub in the repository
\begin{center}
   \url{https://github.com/dmg-lab/regular_flips_in_mptopcom}.
\end{center}

\subsubsection{Acknowledgments}
We are very grateful to Michael Joswig, Benjamin Lorenz, Marta Panizzut, and Francisco Santos Leal for many helpful discussions.
Furthermore we are thankful to J\"org Rambau for the friendly competition.
The author has been supported by MaRDI \cite{mardi} under DFG project ID 460135501.

\printbibliography

@BOOK{Triangulations,
    AUTHOR = {De Loera, Jes{\'u}s A. and Rambau, J{\"o}rg and Santos, Francisco},
     TITLE = {Triangulations},
    SERIES = {Algorithms and Computation in Mathematics},
    VOLUME = {25},
      NOTE = {Structures for algorithms and applications},
 PUBLISHER = {Springer-Verlag},
   ADDRESS = {Berlin},
      YEAR = {2010},
     _PAGES = {xiv+535},
      ISBN = {978-3-642-12970-4},
   MRCLASS = {52B55 (05C10 52B05 57Q15 68U05)},
  MRNUMBER = {2743368 (2011j:52037)},
  DOI = {10.1007/978-3-642-12971-1}
}

@article {mptopcom-paper,
    AUTHOR = {Jordan, Charles and Joswig, Michael and Kastner, Lars},
     TITLE = {Parallel enumeration of triangulations},
   JOURNAL = {Electron. J. Combin.},
  FJOURNAL = {Electronic Journal of Combinatorics},
    VOLUME = {25},
      YEAR = {2018},
    NUMBER = {3},
     PAGES = {Paper 3.6, 27},
      ISSN = {1077-8926},
   MRCLASS = {52B55 (52B70 68U05)},
  MRNUMBER = {3829292},
  arxiv = {1709.04746},
  DOI = {10.37236/7318}
}

@InProceedings{TOPCOM-paper,
  author = 	 {Rambau, Jörg},
  title = 	 {{\TOPCOM}: Triangulations of Point Configurations and Oriented Matroids},
  booktitle = {Mathematical Software --- ICMS 2002},
  year = 	 2002,
  editor = 	 {Cohen, A.M. and Gao, X.-S. and Takayama, N.},
  pages = 	 {330-340},
  publisher = {World Scientific},
  DOI = {10.1142/9789812777171_0035}
}

@incollection {DMV:polymake,
    AUTHOR = {Gawrilow, Ewgenij and Joswig, Michael},
     TITLE = {polymake: a framework for analyzing convex polytopes},
 BOOKTITLE = {Polytopes---combinatorics and computation (Oberwolfach, 1997)},
    SERIES = {DMV Sem.},
    VOLUME = {29},
     PAGES = {43--73},
 PUBLISHER = {Birkh\"auser},
   ADDRESS = {Basel},
      YEAR = {2000},
   MRCLASS = {52B55 (68U05)},
  MRNUMBER = {MR1785292 (2001f:52033)},
  DOI = {10.1007/978-3-0348-8438-9_2}
}

@article{mplrs,   
  author = {David Avis and Charles Jordan},
  title = {\mplrs: A scalable parallel vertex/facet enumeration code},
  year = 2018,    
  journal = {Mathematical Programming Computation},
  doi = {10.1007/s12532-017-0129-y},
  volume = {10},         
  number = {2},                   
  pages = {267--302}  
}

@misc{mardi,
  author       = {{The MaRDI consortium}},
  title        = {{MaRDI: Mathematical Research Data Initiative 
                   Proposal}},
  month        = may,
  year         = 2022,
  publisher    = {Zenodo},
  doi          = {10.5281/zenodo.6552436},
}

@article{polymake:2017,
    author = {Assarf, Benjamin and Gawrilow, Ewgenij and Herr, Katrin and Joswig, Michael and Lorenz, Benjamin and Paffenholz, Andreas and Rehn, Thomas},
     title = {Computing convex hulls and counting integer points with \polymake},
   journal = {Math. Program. Comput.},
    volume = {9},
      year = {2017},
    number = {1},
     pages = {1-38},
   mrclass = {90C57 (52 90-04)},
  mrnumber = {3613012},
       doi = {10.1007/s12532-016-0104-z},
     arxiv = {1408.4653v2},
}

@article{reverse-search,
title = {Reverse search for enumeration},
journal = {Discrete Applied Mathematics},
volume = {65},
number = {1},
pages = {21-46},
year = {1996},
note = {First International Colloquium on Graphs and Optimization},
issn = {0166-218X},
doi = {10.1016/0166-218X(95)00026-N},
author = {David Avis and Komei Fukuda},
}

@Article{schlaefli,
 Author = {Joswig, Michael and Panizzut, Marta and Sturmfels, Bernd},
 Title = {The {Schl{\"a}fli} fan},
 FJournal = {Discrete \& Computational Geometry},
 Journal = {Discrete Comput. Geom.},
 ISSN = {0179-5376},
 Volume = {64},
 Number = {2},
 Pages = {355--381},
 Year = {2020},
 Language = {English},
 DOI = {10.1007/s00454-020-00215-x},
 zbMATH = {7242482},
 Zbl = {1505.14128}
}

@Article{ftheory,
 Author = {Bies, Martin and Cveti{\v{c}}, Mirjam and Donagi, Ron and Liu, Muyang and Ong, Marielle},
 Title = {Root bundles and towards exact matter spectra of {F}-theory {MSSMs}},
 FJournal = {Journal of High Energy Physics},
 Journal = {J. High Energy Phys.},
 ISSN = {1126-6708},
 Volume = {2021},
 Number = {9},
 Pages = {65},
 Note = {Id/No 76},
 Year = {2021},
 Language = {English},
 DOI = {10.1007/JHEP09(2021)076},
 zbMATH = {7415399},
 Zbl = {1472.81199}
}

@Article{tropical_convexity,
 Author = {Develin, Mike and Sturmfels, Bernd},
 Title = {Tropical convexity},
 FJournal = {Documenta Mathematica},
 Journal = {Doc. Math.},
 ISSN = {1431-0635},
 Volume = {9},
 Pages = {1--27},
 Year = {2004},
 Language = {English},
 zbMATH = {2095714},
 Zbl = {1054.52004}
}

@Article{multisplits,
 Author = {Schr{\"o}ter, Benjamin},
 Title = {Multi-splits and tropical linear spaces from nested matroids},
 FJournal = {Discrete \& Computational Geometry},
 Journal = {Discrete Comput. Geom.},
 ISSN = {0179-5376},
 Volume = {61},
 Number = {3},
 Pages = {661--685},
 Year = {2019},
 Language = {English},
 DOI = {10.1007/s00454-018-0021-1},
 zbMATH = {7035808},
 Zbl = {1472.52019}
}

@Book{tropical_essentials,
 Author = {Joswig, Michael},
 Title = {Essentials of tropical combinatorics},
 FSeries = {Graduate Studies in Mathematics},
 Series = {Grad. Stud. Math.},
 ISSN = {1065-7338},
 Volume = {219},
 ISBN = {978-1-4704-6741-8; 978-1-4704-6740-1},
 Year = {2021},
 Publisher = {Providence, RI: American Mathematical Society (AMS)},
 Language = {English},
 DOI = {10.1090/gsm/219},
 zbMATH = {7517377}
}

@Article{shortest_paths,
 Author = {Joswig, Michael and Schr{\"o}ter, Benjamin},
 Title = {Parametric shortest-path algorithms via tropical geometry},
 FJournal = {Mathematics of Operations Research},
 Journal = {Math. Oper. Res.},
 ISSN = {0364-765X},
 Volume = {47},
 Number = {3},
 Pages = {2065--2081},
 Year = {2022},
 Language = {English},
 DOI = {10.1287/moor.2021.1199},
 zbMATH = {7592369},
 Zbl = {1507.90182}
}

@Article{cubic_surfaces,
 Author = {Panizzut, Marta and Vigeland, Magnus Dehli},
 Title = {Tropical lines on cubic surfaces},
 FJournal = {SIAM Journal on Discrete Mathematics},
 Journal = {SIAM J. Discrete Math.},
 ISSN = {0895-4801},
 Volume = {36},
 Number = {1},
 Pages = {383--410},
 Year = {2022},
 Language = {English},
 DOI = {10.1137/20M136520X},
 zbMATH = {7471554}
}

\end{document}